\documentclass{amsart}[12pt]
\usepackage{amsmath,amsfonts,amssymb,latexsym,amsthm}
\usepackage{graphicx,epsf}
\usepackage{epsfig}
\usepackage{amssymb}
\usepackage{amsthm}
\usepackage{multicol}
\usepackage{tikz}
\usetikzlibrary{automata,positioning,arrows}
\usepackage{tikz}
\usetikzlibrary{arrows,matrix,positioning}

\newtheorem{thm}{Theorem}% [section]
\newtheorem{lem}[thm]{Lemma}

\newtheorem{Example}[thm]{Example}

\def\ord{\eta}

\def\bw{{\bold w}}

\def\zz{\mathbb Z}
\def\nn{\mathbb N}

\def\qqq{\mathbb Q}

\def\Ga{{\Gamma}}

\def\la{\lambda}
\def\ga{\gamma}

\def\al{\alpha}
\def\be{\beta}

\def\ca{\mathcal A}

\def\ssu{\subset}

\def\<{\langle}
\def\>{\rangle}

\def\di{\diamond}

\def\ts{\hskip.015cm}

\def\0{{\mathbf 0}}

\def\SL{{\text{\rm SL}}}
\def\GL{{\text{\rm GL}}}
\def\BS{{\text{\rm BS}}}

\def\Cay{{\text {\rm Cay} } }

\def\.{\hskip.06cm}
\def\ts{\hskip.03cm}

\def\nin{\noindent}

\begin{document}

\title[Linear groups and P-recursiveness]{Words in linear groups, random walks,
automata \\ and P-recursiveness}

\author[Scott Garrabrant \and Igor Pak]{Scott~Garrabrant$^{\star}$ \and Igor~Pak$^{\star}$}

\thanks{\thinspace ${\hspace{-.45ex}}^\star$Department of Mathematics,
UCLA, Los Angeles, CA, 90095.
\hskip.06cm
Email:
\hskip.06cm
\texttt{\{coscott,\ts{pak}\}@math.ucla.edu}}

\maketitle

{\hskip5.3cm
\today
}

\vskip1.3cm

\begin{abstract}
Fix a finite set $S \ssu \GL(k,\zz)$.  Denote by $a_n$ the number
of products of matrices in $S$ of length~$n$ that are equal to~1.
We show that the sequence $\{a_n\}$ is not always P-recursive.
This answers a question of Kontsevich.
\end{abstract}

\vskip.4cm

\hskip3.3cm

\vskip.9cm

\section{Introduction}

\medskip

\nin
An integer sequence $\{a_n\}$ is called \emph{polynomially recursive}, or \emph{P-recursive},
if it satisfies a nontrivial linear recurrence relation of the form
$$
(\ast) \qquad
q_0(n)\ts a_n \. + \.  q_1(n)\ts a_{n-1} \. + \. \ldots \. + \. q_k(n)\ts a_{n-k} \. = \.0\.,
$$
for some $q_i(x) \in \zz[x]$, $0\le i \le k$.  The study of P-recursive sequences
plays a major role in modern Enumerative and Asymptotic Combinatorics,
see e.g.~\cite{FS,Ges,Odl,Sta}.  They have \emph{D-finite} (also called \emph{holonomic})
generating series
$$
\ca(t) \. = \. \sum_{n=0}^\infty \. a_n \ts t^n\.,
$$
and various asymptotic properties (see Section~\ref{s:prob} below).

Let $G$ be a group and $\zz[G]$ denote its group ring.  For every
$g\in G$ and $u \in \zz[G]$, denote by $[g] \.u$ the value of $u$ on~$g$.
Let $a_n = [1]\ts u^n$, which denotes the value of $u^n$ at the
identity element.  When $G=\zz^k$ or $G=F_k$, the sequence $\{a_n\}$
is known to be P-recursive for all $u \in \zz[G]$, see~\cite{Hai}.
Maxim Kontsevich asked whether $\{a_n\}$ is always
P-recursive when $G\subseteq\GL(k,\zz)$, see~\cite{Sta-mfo}.
We give a negative answer to this question:

\begin{thm}\label{t:main}
There exists an element $u \in \mathbb{Z}[\SL(4,\mathbb{Z})]$, such that
the sequence $\{\. [1]\ts u^n\ts\}$ is not \ts P-recursive.
\end{thm}

We give two proofs of the theorem.  The first proof is completely
self-contained and based on ideas from computability.  Roughly, we
give an explicit construction of a finite state automaton with two
stacks and a non-P-recursive sequence of accepting path lengths
(see Section~\ref{s:auto}).  We then convert this automaton into
a generating set~$S \ssu \SL(4,\zz)$, see Section~\ref{s:proof}.
The key part of the proof is a new combinatorial lemma giving
an obstruction to P-recursiveness (see Section~\ref{s:parity}).

\smallskip

Our second proof of Theorem~\ref{t:main} is analytic in nature,
and is the opposite of being self-contained.
We interpret the problem in a probabilistic language, and use
a number of advanced and technical results in Analysis, Number Theory,
Probability and Group Theory to derive the theorem.
Let us briefly outline the connection.

Let~$S$ be a generating set of the group~$G$.  Denote by $p(n) = p_{G,S}(n)$
the probability of return after $n$ steps of a random walk
on the corresponding Cayley graph $\Cay(G,S)$.  Finding the
asymptotics of~$p(n)$ as $n\to \infty$ is a fundamental problem
in probability, with a number of both classical and recent results
(see e.g.~\cite{Pete,Woe}).
In the notation above, we have:
$$
p(n) \. = \. \frac{a_n}{|S|^n}\,, \quad \text{where} \quad  a_n \. =
\. [1] \ts u^n \quad \text{and} \quad
u \. = \. \sum_{s\in S} \. s\ts.
$$
Since P-recursiveness of $\{a_n\}$ implies P-recursiveness of~$\{p(n)\}$,
and much is known about the asymptotic of both $p(n)$ and P-recursive sequences,
this connection can be exploited to obtain non-P-recursive examples
(see Section~\ref{s:prob}).  See also Section~\ref{s:fin} for
final remarks and historical background behind the two proofs.

\bigskip

\section{Parity of P-Recursive Sequences}\label{s:parity}

\medskip

\nin
In this section, we give a simple obstruction to P-recursiveness.

\begin{lem}\label{l:mod2}
Let $\{a_n\}$ be a P-recursive integer sequence. Consider an
infinite binary word ${\bold w}=w_1w_2\ldots$ defined by
$w_n= a_n\,\,{\rm mod}\,\, 2$.  Then, there exists a finite
binary word~$v$ which is not a subword of~$w$.
\end{lem}

\begin{proof}
Let $\ord(n)$ denote the largest integer~$r$ such that~$2^r|n$.
By definition, there exist polynomials \ts
$q_0,\ldots,q_k\in\zz[n]$, such that
$$
a_n\,= \, \frac{1}{q_0(n)} \bigl(a_{n-1}\ts q_1(n)\. + \. \ldots\.
+\. a_{n-k} \ts q_k(n)\bigr), \ \ \, \text{for all} \ \, n> k\ts.
$$
% Without loss of generality, we may assume that $p_0,\ldots,p_k\in\mathbb{Z}[n]$.
% Indeed, since $A_n$ is an integer sequence, the coefficients in $p_i$ are a
% nontrivial solution to a system of linear equations over $\mathbb{Q}$.
% Thus, there must exist a nontrivial solution to this system within $\mathbb{Q}$,
% so we may assume that $p_0,\ldots,p_k\in\mathbb{Q}[n]$.
% Multiplying all of the polynomials by a large integer allows us to assume that
% $p_0,\ldots,p_k\in\mathbb{Z}[n]$.

Let $\ell$ be any integer such that \ts $q_i(\ell)\neq 0$ \ts for all~$i$.
Similarly, let~$m$ be the smallest integer such that \ts $2^m>k$, and \ts $m>\ord(q_i(\ell))$
for all~$i$. Finally, let $d>0$ be such that
$\ord(q_d(\ell))\leq \ord(q_i(\ell))$ for all $i>0$.

Consider all $n$ such that:
\begin{equation}
n=\ell~\text{mod}~2^m,\quad w_{n-d}=1 \quad \text{and} \quad w_{n-i}=0\ \ \, \text{for all} \ \  i\neq 0,\ts d\ts.\tag{$\star$}
\end{equation}
Note that $\ord(q_i(n))=\ord(q_i(\ell))$ for all~$i$,
since $q_i(n)= q_i(\ell)\mbox{ mod }2^m$ and $\ord(q_i(\ell))<m.$
We have
$$
\ord(a_n)=\ord\Big(a_{n-1}q_1(\ell)+\ldots+a_{n-k}q_k(\ell)\Big)-\ord(q_0(\ell)).
$$
Since $\ord(a_{n-d}q_d(\ell))<\ord(a_{n-i}q_i(\ell))$ for all $i\neq d,$ this implies that
$$
\ord(a_n)=\ord(a_{n-d}q_d(\ell))-\ord(q_0(\ell))=\ord(q_d(\ell))-\ord(q_0(\ell)).
$$
Therefore, $w_n=1$ if and only if $\ord(q_d(\ell))=\ord(q_0(\ell))$.
This implies that~$w_n$ is independent of~$n$, and must be the same for all~$n$
satisfying~$(\star)$. In particular, this means that at least one of the words
$0^{k-d}10^{d-1}1$ and $0^{k-d}10^{d}$ cannot appear in ${\bold w}$ ending at
a location congruent to $\ell$ modulo~$2^m$.

Consider the word $v=(0^{k-d} 1 {0^k} 1 0^{d-1})^{2^m}$. Note that
$0^{k-d} 1 {0^k} 1 0^{d-1}$ has odd length, and contains both
$0^{k-d}10^{d-1}1$ and $0^{k-d}10^{d}$ as subwords.
Therefore, the word~$v$ contains both $0^{k-d}10^{d-1}1$ and $0^{k-d}10^{d}$
in every possible starting location modulo $2^m$.
This implies that~$v$ cannot appear as a subword of ${\bold w}$.
\end{proof}

\bigskip

\section{Building an Automaton} \label{s:auto}

\medskip

\nin
In this section we give an explicit construction of
a finite state automaton with the number of accepting
paths given by a binary sequence which does not satisfy
conditions of Lemma~\ref{l:mod2}.

\smallskip

Let $X\simeq F_3$ be the free group generated by $x$, $1_x$, and $0_x$.
Similarly, let $Y\simeq F_3$ be the free group generated by $y$, $1_y$, and $0_y$.
We assume that $X$ and~$Y$ commute.

Define a directed graph $\Gamma$ on vertices $\{s_1,\ldots,s_8\}$, and
with edges as shown in Figure~1.
Some of the edges in $\Gamma$ are labeled with elements of $X$, $Y$,
or both.
For a path $\gamma$ in $\Gamma$, denote by $\omega_X(\gamma)$
the product of all elements of $X$  in $\gamma$, and by
$\omega_Y(\gamma)$ denote the product of all elements of $Y$ in~$\gamma$.
By a slight abuse of notation, while traversing~$\gamma$ we will
use $\omega_X$ and $\omega_Y$ to refer to the product of all elements of~$X$
and~$Y$, respectively, on edges that have been traversed so far.

Finally, let $b_n$ denote the number of paths in $\Gamma$ from $s_1$ to $s_8$
of length~$n$, such that $\omega_X(\gamma)=\omega_Y(\gamma)=1$.
For example, the path
$$
\ga\.: \ \,
s_1\xrightarrow{\scriptsize xy} s_1\rightarrow s_2\xrightarrow{1_yx^{-1}}
s_4\xrightarrow{1_y^{-1}1_x} s_4\xrightarrow{y^{-1}} s_5\rightarrow s_6\xrightarrow{1_x^{-1}}s_8
$$
is the unique such path of length~$7$, so $b_7=1$.

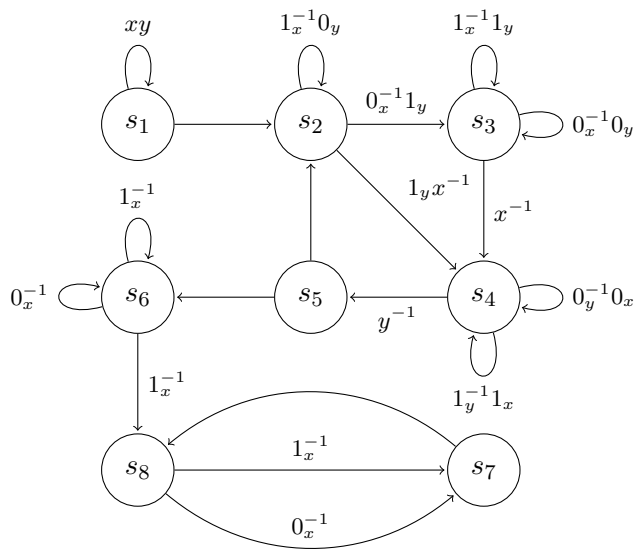
\begin{figure}[hbt]
\label{gamma}
\tiny
\begin{tikzpicture}[shorten >=1pt,node distance=2.3cm,on grid,auto]
      \large
      \node[state] (a)   {$s_1$};
   \node[state] (b) [right=of a] {$s_2$};
   \node[state] (c) [right=of b] {$s_3$};
   \node[state](d) [below =of c] {$s_4$};
   \node[state](e) [left=of d] {$s_5$};
   \node[state](f) [left=of e] {$s_6$};
   \node[state](g) [below=of d] {$s_7$};
   \node[state](h) [below=of f] {$s_8$};
    \path[->]
    (a) edge  node {} (b)
          edge  [loop above] node {\small $xy$} ()
    (b) edge  node  {\small $0_x^{-1}1_y$} (c)
    edge  node  {\small $1_yx^{-1}$} (d)
          edge [loop above] node {\small $1_x^{-1}0_y$} ()
    (c) edge  node {\small $x^{-1}$} (d)
          edge [loop above] node {\small $1_x^{-1}1_y$} ()
          edge [loop right] node {\small $0_x^{-1}0_y$} ()
    (d) edge  node {\small $y^{-1}$} (e)
          edge [loop below] node {\small $1_y^{-1}1_x$} ()
          edge [loop right] node {\small $0_y^{-1}0_x$} ()
    (e) edge node {} (b)
          edge node {} (f)
    (f) edge  node {\small $1_x^{-1}$} (h)
          edge [loop above] node {\small $1_x^{-1}$} ()
          edge [loop left] node {\small $0_x^{-1}$} ()
    (g) edge [bend right=40] node {} (h)
    (h) edge node {\small $1_x^{-1}$} (g)
         edge [bend right=40] node {\small $0_x^{-1}$} (g);
\end{tikzpicture}
\caption{The graph $\Gamma$.}
\end{figure}

\begin{lem}
For every $n\ge 1$ we have $b_n \in \{0,1\}$.
Moreover, every finite binary word
is a subword of \ts ${\bold b}=b_1b_2\ldots$
\end{lem}

\begin{proof}
To simplify the presentation, we split the proof into two parts.

\medskip

\nin
{\bf {\small (a)}~The structure of paths.} \ts
Let $\gamma$ be a path from $s_1$ to~$s_8$.  Denote by~$k$ the number
of times $\gamma$ traverses the loop $s_1\xrightarrow{xy}s_1$.
The value of $\omega_X$ after traversing these $k$ loops is $x^k$,
and the value of $\omega_Y$ is $y^k$.
	
There must be $k$ instances of the edge $s_4\xrightarrow{y^{-1}}s_5$
in $\gamma$ to cancel out the $y^k$. Further, any time the path
traverses this edge, the product~$\omega_Y$ must change from some $y^j$ to $y^{j-1}$,
with no $0_y$ or $1_y$ terms. Therefore, every time $\gamma$ enters
the vertex $s_4$, it must traverse the two loops
$s_4\xrightarrow{1_y^{-1}1_x}s_4$ and $s_4\xrightarrow{0_y^{-1}0_x}s_4$
enough to replace any $0_y$ and $1_y$ terms in $\omega_Y$ with $0_x$
and $1_x$ terms in $\omega_X$. This takes the binary word at the end
of $\omega_Y$, and moves it to the end of $\omega_X$ in the reverse order.

Similarly, any time $\gamma$ traverses the edge $s_3\xrightarrow{x^{-1}}s_4$
or $s_2\xrightarrow{1_yx^{-1}}s_4$, the product~$\omega_X$ must change from some
$x^j$ to $x^{j-1}$, with no $0_x$ or $1_x$ terms.
Every time $\gamma$ enters the vertex $s_2$, it must remove all
$0_x$ and $1_x$ terms from $\omega_X$ before transitioning to $s_4$.
The $s_2$ and $s_3$ vertices ensure that as this binary word is deleted
from $\omega_X$, another binary word is written at the end of $\omega_Y$
such that the reverse of the binary word written at the end of $\omega_Y$
is one greater as a binary integer than the word removed from the end
of~$\omega_X$.
	
Every time $\gamma$ traverses the edge $s_4\xrightarrow{y^{-1}}s_5$,
the number written in binary at the end of $\omega_X$ is incremented by one.
Thus, after traversing this edge $k$ times, the $X$ word will consist
of $k$ written in binary, and $\omega_Y$ will be the identity.
At this point, $\gamma$ will traverse the edge $s_5\xrightarrow{y^{-1}}s_6$.
	
After entering the vertex $s_6$, all of the $0_x$ and $1_x$ terms
from $\omega_X$ will be removed. Each time a $1_x$ term is removed,
$\gamma$ can move to the vertex $s_8$. From $s_8$, the $0_x$ and $1_x$
terms will continue to be removed, but $\gamma$ will traverse two edges
for every term removed, thus moving at half speed. After all of these
terms are removed, the products $\omega_X(\gamma)$ and $\omega_Y(\gamma)$
are equal to identity, as desired.

\smallskip

\nin
{\bf {\small (b)}~The length of paths.} \ts
Now that we know the structure of paths through~$\Gamma$,
we are ready to analyze the possible lengths of these paths.
There are only two choices to make in specifying a path $\gamma$~:
first, the number $k=k(\gamma)$ of times the loop from $s_1$ to
itself is traversed, and second, the number $j=j(\gamma)$ of digits
still on $\omega_X(\gamma)$ immediately before traversing the edge
from $s_6$ to $s_8$. The number $j$ must be such that the $j$-th
binary digit of $k$ is a~1.
	
When $\gamma$ reaches $s_5$ for the first time, it has traversed $k+4$ edges.
In moving from the $i$-th instance of $s_5$ along $\gamma$ to the
$(i+1)$-st instance of $s_5$, the number of edges traversed is
$3+\lfloor 1+\log_2(i)\rfloor+\lfloor 1+\log_2(i+1)\rfloor$,
three more than the sum of the number of binary digits in $i$ and $i+1$.
Therefore, the number of edges traversed by the time $\gamma$ reaches~$s_6$
is equal to
$$
k+5+\sum_{i=1}^{k-1}(3+\lfloor 1+\log_2(i)\rfloor+\lfloor 1+\log_2(i+1)\rfloor).
$$

If $j=1$, the edge from $s_6$ to $s_8$ is traversed at the last possible
opportunity and $\lfloor 1+\log_2(k)\rfloor$ more edges are traversed.
However, if $j>1$, there are an additional $j-1$ edges traversed,
since the $s_7$ and $s_8$ states do not remove $\omega_X$ terms as
efficiently as $s_6$. In total, this gives
$|\gamma|=L(k(\gamma),j(\gamma))$, where
$$
L(k,j) \. = \. j-1+ \ts \lfloor 1+\log_2(k)\rfloor \ts +k+5+ \. 
\sum_{i=1}^{k-1}\.\Bigl(3+\ts \lfloor 1+\log_2(i)\rfloor\ts +\ts \lfloor 1+\log_2(i+1)\rfloor\Bigr).
$$	
This simplifies to
$$
L(k,j) \. = \. j+6k+ 2\ts \sum_{i=1}^{k}\.\lfloor \log_2i\rfloor\..
$$
Since $1\leq j\leq \lfloor 1+\log_2(k)\rfloor$, we have $L(k+1,1)>L(k,j)$
for all possible values of~$j$. Thus, there are no two paths of the same length,
which proves the first part of the lemma.

Furthermore, we have $b_n=1$ if and only if $n=L(k,j)$
for some $k\geq1$ and $j$ such that the $j$-th binary digit of $k$ is a~1.
Thus, the binary subword of ${\bold b}$ at locations $L(k,1)$ through
$L(k,\lfloor{1+\log_2(k)}\rfloor)$ is exactly the integer~$k$ written in binary.
This is true for every positive integer $k$, so ${\bold b}$ contains
every finite binary word as a subword.
\end{proof}

\begin{Example} {\rm 
For $k=3$ and $j=2$, we have $L(k,j)=24$. This corresponds to the
unique path in~$\Gamma$ of length~24:
$$
\aligned 
s_1\xrightarrow{xy} s_1 \xrightarrow{xy} s_1\xrightarrow{xy} s_1\rightarrow s_2
\xrightarrow{1_yx^{-1}} s_4\xrightarrow{1_y^{-1}1_x} s_4\xrightarrow{y^{-1}} s_5\rightarrow s_2 
\hskip2.cm \\
\xrightarrow{1_x^{-1}0_y}s_2\xrightarrow{1_yx^{-1}} s_4\xrightarrow{1_y^{-1}1_x} s_4
\xrightarrow{0_y^{-1}0_x} s_4\xrightarrow{y^{-1}} s_5\rightarrow s_2
\xrightarrow{0_x^{-1}1_y}s_3\xrightarrow{1_x^{-1}1_y} s_3
\\ 
\hskip1.7cm
\xrightarrow{x^{-1}} s_4\xrightarrow{1_y^{-1}1_x} s_4\xrightarrow{1_y^{-1}1_x} s_4
\xrightarrow{y^{-1}} s_5\rightarrow s_6\xrightarrow{1_x^{-1}}s_8
\xrightarrow{1_x^{-1}}s_7\rightarrow s_8\ts.
\endaligned
$$
}
\end{Example}

\bigskip

\section{Proof of Theorem~\ref{t:main}} \label{s:proof}

\subsection{From automata to groups}
We start with the following technical lemma.

\begin{lem}\label{l:DFA}
Let $G=F_{11}\times F_{3}$. Then there exists an element $u\in \mathbb{Z}[G]$,
such that \ts $[1]\ts u^{2n+1}$ \ts is always even, and
${\bold w}=w_1w_2\ldots$ given by
$w_n= \left(\frac{1}{2} \. [1]\ts u^{2n+1}\right)\,{\rm mod}\, 2$, is an infinite binary word
that contains every finite binary word as a subword.
\end{lem}

\begin{proof}
We suggestively label the generators of $F_{11}$ as
$\{s_1,s_2,s_3,s_4,s_5,s_6,s_7,s_8,x,0_x,1_x\}$ and
label the generators of $F_3$ as $\{y,0_y,1_y\}$.
Consider the following set $S$ of 19 elements of $G$:
\begin{multicols}{3}
\begin{enumerate}
\item $z_1=s_1^{-1}xys_1$,
\item $z_2=s_1^{-1}s_2$,
\item $z_3=s_2^{-1}1_x^{-1}0_ys_2$,
\item $z_4=s_2^{-1}0_x^{-1}1_ys_3$,
\item $z_5=s_3^{-1}1_x^{-1}1_ys_3$,
\item $z_6=s_3^{-1}0_x^{-1}0_ys_3$,
\item $z_7=s_3^{-1}x^{-1}s_4$,
\item $z_8=s_2^{-1}1_yx^{-1}s_4$,
\item $z_9=s_4^{-1}1_y^{-1}1_xs_4$,
\item $z_{10}=s_4^{-1}0_y^{-1}0_xs_4$,
\item $z_{11}=s_4^{-1}y^{-1}s_5$,
\item $z_{12}=s_5^{-1}s_2$,
\item $z_{13}=s_5^{-1}s_6$,
\item $z_{14}=s_6^{-1}1_x^{-1}s_6$,
\item $z_{15}=s_6^{-1}0_x^{-1}s_6$,
\item $z_{16}=s_6^{-1}1_x^{-1}s_8$,
\item $z_{17}=s_7^{-1}s_8,$
\item $z_{18}=s_8^{-1}1_x^{-1}s_7$,
\item $z_{19}=s_8^{-1}0_x^{-1}s_7$.
\end{enumerate}
\end{multicols}

Let $\Ga$ be as defined in the previous section.
For every edge from $s_i\xrightarrow{r}s_j$ in $\Gamma$, there is one element
of $S$ equal to $s_i^{-1}rs_j$. We show that the number of ways to
multiply $n$ terms from $S$ to get $s_1^{-1}s_8$ is exactly $b_n$.

First, we show that there is no product of terms in $S$ whose $F_{11}$ component is the identity. Assume that such a product exists,
and take one of minimal length. If there are two consecutive terms in this
product such that $s_i$ at the end of one term does not cancel the
$s_j^{-1}$ at the start of the following term, then either the $s_i$
must cancel with a $s_i^{-1}$ before it or the $s_j^{-1}$ must cancel
with a $s_j$ after it. In both cases, this gives a smaller sequence of
terms whose product must have $F_{11}$ component equal to the identity.
If the $s_i$ at the end of each term cancels the $s_j^{-1}$ at
the beginning of the next term, then this product corresponds
to a cycle $\gamma\in\Gamma$ such that $\omega_X(\gamma)$ is the identity.
Straightforward analysis of $\Gamma$ shows that no such cycle exists,
so there is no product of terms in $S$ whose product $F_{11}$ component
equal to the identity.

This also means that the $s_i$ at the end of each term must cancel
the $s_j^{-1}$ at the start of the following term, since otherwise
ether the $s_i$ must cancel with a $s_i^{-1}$ before it or the $s_j^{-1}$
must cancel with a $s_j$ after it, forming a product of terms in $S$
whose $F_{11}$ component is equal to the identity.

Since each $s_i$ cancels with an $s_i^{-1}$ at the start of the following term,
the product must correspond to a path $\gamma\in\Gamma.$ If $\gamma$
is from $s_i$ to $s_j$, the product will evaluate to
$s_i^{-1}\omega_X(\gamma)\omega_Y(\gamma)s_j$.
Therefore, the number of ways to multiply $n$ terms from $S$ to get
$s_1^{-1}s_8$ is equal to $b_n$.

\smallskip
We can now define $u\in \zz[G]$ as
$$
u\. =\. 2s_8^{-1}s_1 \. + \. \sum_{{z_i}\in S} \. z_i\ts.
$$
We claim that $\frac{1}{2} \.[1]\ts u^{2n+1}\ts=\ts b_{2n}\mbox{ mod }2$.
We already showed that one cannot get $1$ by multiplying only elements of~$S$,
so the $2s_8^{-1}s_1$ term must be used at least once.
If this term is used more than once, then the contribution to
$[1]\ts u^{2n+1}$ will be $0 \mbox{ mod }4$.
Therefore, we need only consider the cases where this term is used exactly once,
so $\frac{1}{2}\.[1]\ts u^{2n+1}$ is equal modulo $2$ to the number of products
of the form
\begin{equation}
2=z_{i_1}\ldots z_{i_{k-1}}(2s_8^{-1}s_1)z_{i_{k+1}}\ldots z_{i_{2n+1}}.
\tag{$\star\star$}
\end{equation}
This condition holds if and only if
$$z_{i_{k+1}}\ldots z_{i_{2n+1}}z_{i_1}\ldots z_{i_{k-1}}=s_1^{-1}s_8,$$
which can be achieved in $b_{2n}$ ways.

There are $2n+1$ choices for the location $k$ of the $2s_8^{-1}s_1$ term,
and for each such $k$, there are $b_{2n}$ solutions to $(\star\star)$.
This gives
$$
\frac{1}{2} \. [1]\ts u^{2n+1}\. =\. (2n+1)\ts  b_{2n} \. =\.  b_{2n} \mod 2\ts,
$$
which implies $w_n= b_{2n}$.  By Lemma~$\ref{l:DFA}$, we conclude that ${\bold w}$ is
an infinite binary word which contains every finite binary word as a subword.
\end{proof}

\medskip

\subsection{Counting words mod~2}
We first deduce the main result of this paper and then give a useful minor extension.

\begin{proof}[Proof of Theorem~\ref{t:main}]
The group $\SL(4,\mathbb{Z})$ contains $\SL(2,\mathbb{Z})\times \SL(2,\mathbb{Z})$
as a subgroup. The group $\SL(2,\mathbb{Z})$ contains Sanov's subgroup isomorphic
to~$F_2$, and thus every finitely generated free group $F_\ell$ as a subgroup
(see e.g.~\cite{dlH}).  Therefore, $F_{11}\times F_{3}$ is a subgroup of
$\SL(4,\mathbb{Z})$, and the element $u\in\mathbb{Z}[F_{11}\times F_{3}]$ defined in
Lemma~\ref{l:DFA} can be viewed as an element of $\mathbb{Z}[\SL(4,\mathbb{Z})]$.

Let $a_n=[1]\ts u^n$.  By Lemma~\ref{l:DFA}, the number $a_{2n+1}$ is always even,
and the word ${\bold w}=w_1w_2\ldots$ given by \ts
$w_n= \frac{1}{2} \. a_{2n+1}\mbox{ mod }2$ \ts is an infinite binary word which
contains every finite binary word as a subword. Therefore, by Lemma~\ref{l:mod2},
the sequence $\bigl\{\frac{1}{2} \.a_{2n+1}\bigr\}$ is not P-recursive. Since P-recursivity is
closed under taking a subsequence consisting of every other term, the sequence
$\{a_n\}$ is also not P-recursive.
\end{proof}

\begin{thm}\label{t:two-gen-sets}
There is a group $G\ssu \SL(4,\zz)$ and two generating sets $\<S_1\> = \<S_2\> = G$,
such that for the elements
$$
u_1 \. = \. \sum_{s\in S_1} \. s \,, \qquad  u_2 \. = \. \sum_{s\in S_2}\. s \,,
$$
we have the sequence $\{[1]\ts u_1^n\}$ is P-recursive,
while $\{[1]\ts u_2^n\}$ is not P-recursive.
\end{thm}

\begin{proof}
Let $G = F_{11} \times F_3$ be as above.  Denote by $X_1$ and~$X_2$ the standard
generating sets of $F_{11}$ and~$F_3$, respectively.  Finally, let
$S_1 = (X\times 1) \cup (1 \times Y)$,
$$w_1=\sum_{x\in X_1} \ts x, \quad w_2=\sum_{x\in X_2} \ts x\ts.
$$
Recall that if $\{c_n\}$ is P-recursive, then so is $\{c_n/n!\}$ and
$\{c_n \cdot n!\}$.  Observe that
$$
\sum_{n=0}^\infty \. [1]\ts u_1^n \ts \frac{t^n}{n!} \, = \,
\left(\sum_{n=0}^\infty \. [1]\ts w_1^n \ts \frac{t^n}{n!}\right)
\left(\sum_{n=0}^\infty \. [1]\ts w_2^n \ts \frac{t^n}{n!}\right),
$$
and that $\{\ts [1]\ts w_1^n\ts\}$ and $\{\ts [1]\ts w_2^n\ts\}$
are P-recursive by Haiman's theorem~\cite{Hai}.  This implies that
$\{\ts [1]\ts u_1^n\ts\}$ is also P-recursive, as desired.

Now, let $S_2 = 2\ts S_1 \cup S$, where $S$ is the set constructed in
the proof of Lemma~\ref{l:DFA}, and $2\ts S_1$ means that each element
of~$S_1$ is taken twice.  Observe that \ts
$[1]\ts u_2^n \. = \. [1]\ts u^n\mbox{ mod }2$, where $u$ is as in 
the proof of Theorem~\ref{t:main}.  This implies that
$\{\ts [1]\ts u_1^n\ts\}$ is not P-recursive,
and finishes the proof.
\end{proof}

\bigskip

\section{Asymptotics of P-recursive sequences and the return probabilities}\label{s:prob}

\bigskip

\subsection{Asymptotics}\label{ss:prob-asy}
The asymptotics of general P-recursive sequences is undersood to be a finite
sum of the terms
$$
A \. (n!)^s \. \la^n \. e^{Q(n^\ga)} \. n^\al \. (\log n)^\be\ts,
$$
where $s,\ga \in \qqq$, $\al, \la \in \overline \qqq$,
$\be \in \nn$, and $Q(\cdot)$ is a polynomial.
This result goes back to Birkhoff and
Trjitzinsky~(1932), and also Turrittin~(1960).  Although
there are several gaps in these proofs, they are closed now,
notably in~\cite{Imm}.
We refer to~\cite[$\S$VIII.7]{FS}, \cite[$\S$9.2]{Odl} and~\cite{Pak}
for various formulations of general asymptotic estimates,
an extensive discussion of priority issues and further references.

For the integer P-recursive sequences which grow at most exponentially,
the asymptotics have further constraints summarized in the following theorem.

\begin{thm}% [Many authors, see below]
\label{t:asy}
Let $\{a_n\}$ be an integer P-recursive sequence defined by $(\ast)$, and such that
$a_n < C^n$ for some $C>0$ and all $n \ge 1$.  Then
$$
a_n \, \sim \, \sum_{i=1}^m \. A_i \. \la_i^n \. n^{\al_i} \. (\log n)^{\be_i}\,,
$$
where $\al_i \in \qqq$, $\la_i \in\overline{\qqq}$ and $\be_i \in \nn$.
\end{thm}

The theorem is a combination of several known results.  Briefly, the generating
series $\ca(t)$ is a $G$-functions in a sense of Siegel~(1929),
which by the works of  Andr\'{e}, Bombieri, Chudnovsky, Dwork and Katz,
must satisfy an ODE which has only regular singular points and rational
exponents (see a discussion on~\cite[p.~719]{And} and
an overview in~\cite{Beu}).  We then apply the Birkhoff--Trjitzinsky theorem,
which in the regular case has a complete and self-contained proof
(see Theorem~VII.10 and subsequent comments in~\cite{FS}).  We refer
to~\cite{Pak} for further references and details.

\smallskip

\subsection{Probability of return}\label{ss:prob-return}
Let $G$ be a finitely generated group.  A generating set $S$ is called
\emph{symmetric} if $S=S^{-1}$.  Let $H$ be a subgroup of $G$ of finite
index.  It was shown by Pittet and 
Saloff-Coste~\cite{PS1}, that for two symmetric generating sets
$\<S\>=G$ and $\<S'\> = H$ we have
$$(\di)  \qquad
C_1 \ts p_{G,S}(\al_1 n) \. < \. p_{G,S'}(n) \.
< \. C_2 \ts p_{G,S}(\al_2 n)\ts,
$$
for all $n>0$ and fixed constants $C_1,C_2,\al_1,\al_2 >0$.
For $G=H$, this shows, qualitatively, that the asymptotic behavior 
of $p_{G,S}(n)$ is a property of a group.  
The following result gives a complete answer
for a large class of groups.

\begin{thm} %[Many authors, see below]
\label{t:solv}
Let $G$ be an amenable subgroup of $\GL(k,\zz)$ and $S$ is a
symmetric generating set.  Then either $G$ has polynomial growth
and polynomial return probabilities:
$$
A_1 \ts n^{-d} \.  <  \. p_{G,S}(2n) \. < \. A_2 \ts n^{-d}\ts,
$$
or $G$ has  exponential growth and mildly exponential return
probabilities:
$$
A_1 \ts \rho_1^{\sqrt[3]{n}} \. < \. p_{G,S}(2n) \. < \. A_2 \ts \rho_2^{\sqrt[3]{n}}\ts,
$$
for some $A_1,A_2>0$, $0< \rho_1,\rho_2<1$, and $d\in \nn$.
\end{thm}

The  theorem is again a combination of several known results.  Briefly,
by the Tits alternative, group $G$ must be virtually solvable, which implies that
it either has a polynomial or exponential growth (see e.g.~\cite{dlH}).
By the \emph{quasi-isometry}~$(\di)$, we can assume that~$G$ is solvable.  In the
polynomial case, the lower bound follows from the CLT by Cr\'{e}pel and Raugi~\cite{CR}, 
while the upper bound was proved by Varopoulos using the
Nash inequality~\cite{V1} (see also~\cite{V2}).  
For the more relevant to us case of exponential
growth, recall Mal'tsev's theorem,  which says that
all solvable subgroups of $\SL(n,\zz)$ are polycyclic
(see e.g.~\cite[Thm.~22.7]{Sup}). For polycyclic groups of exponential growth,
the upper bound is due to Varopoulos~\cite{Var} and the lower bound is due to
Alexopoulos~\cite{Ale}.  We refer to~\cite{PS} and~\cite[$\S$15]{Woe}
for proofs and further references, and to~\cite{PS-proc} for a generalization
to discrete subgroups of groups of Lie type.  

\smallskip

\subsection{Applications to P-recursiveness}\label{ss:prob-P-rec}
We can now show that non-P-recursiveness for amenable linear groups
of exponential growth.

\begin{thm}  \label{t:prob}
Let $G$ be an amenable subgroup of $\GL(k,\zz)$ of exponential growth,
and let $S$ be a symmetric generating set. Then the probability of return
sequence $\bigl\{p_{G,S}(n)\bigr\}$ is not P-recursive.
\end{thm}

\begin{proof}
It is easy to see that $H$ has exponential growth, so Theorem~\ref{t:solv}
applies.  Let $a_n=|S|^n\ts p_{G,S}(n) \in \nn$ as in the introduction.
If $\{p_{G,S}(n)\}$ is P-recursive, then so is~$\{a_{2n}\}$.  On the other hand,
Theorem~\ref{t:asy} forbids mildly exponential terms $\rho^{\sqrt[3]{n}}$
in the asymptotics of~$a_{2n}$, giving a contradiction.
\end{proof}

To obtain Theorem~\ref{t:main} from here, consider the following linear
group $H \ssu \SL(3,\zz)$ of exponential growth:
$$
H \ = \
\left\{ \ \begin{pmatrix}
x_{1,1} & x_{1,2} & y_{1} \\
x_{2,1} & x_{2,2} & y_{2} \\
0 & 0 & 1 \end{pmatrix}
\quad \text{s.t.} \ \ \begin{pmatrix}
x_{1,1} & x_{1,2} \\
x_{2,1} & x_{2,2} \end{pmatrix} =
\begin{pmatrix}
2 & 1 \\
1 & 1 \end{pmatrix}^k, \ k \in \zz
\ \right\}
$$
(see e.g.~\cite[$\S$15.B]{Woe}).
Observe that $H \simeq \zz \ltimes\zz^2$, and therefore solvable.
Thus, $H$ has a natural symmetric generating set
$$
E \ = \ \left\{ \ \begin{pmatrix}
2 & 1 & 0 \\
1 & 1 & 0 \\
0 & 0 & 1 \end{pmatrix}^{\pm 1},  \
\begin{pmatrix}
1 & 0 & \pm 1 \\
0 & 1 & 0 \\
0 & 0 & 1 \end{pmatrix}, \
\begin{pmatrix}
1 & 0 & 0 \\
0 & 1 & \pm 1 \\
0 & 0 & 1 \end{pmatrix}
\ \right\}.
$$
By Theorem~\ref{t:prob}, the probability of return
sequence $\bigl\{p_{H,E}(n)\bigr\}$ is not P-recursive,
as desired.

\bigskip

\section{Final Remarks}\label{s:fin}

\subsection{}\label{ss:fin-kon}
Kontsevich's question was originally motivated by related questions
on the ``categorical entropy''~\cite{DHKK}.  In response to the
draft of this paper, Ludmil Katzarkov, Maxim Kontsevich and Richard
Stanley asked us if the examples we construct satisfy
\emph{algebraic differential equations} (ADE),
see e.g.~\cite[Exc.~6.63]{Sta}.  We believe that the answer is No,
and plan to explore this problem in the future.

\subsection{}\label
{ss:fin-undecide}
\ts The motivation behind the proof of Theorem~\ref{t:main} lies in the classical
result of Miha{\u\i}lova that $G=F_2 \times F_2$ has an undecidable
group membership problem~\cite{Mih}.  In fact, we conjecture that the problem
whether $\bigl\{\.[1] \. u^n\bigr\}$ is P-recursive is undecidable.
We refer to~\cite{Hal} for an extensive
survey of decidable and undecidable matrix problems.

\subsection{}\label{ss:fin-group}
\ts  Following the approach of the previous section,
Theorem~\ref{t:prob} can be extended to all
polycyclic groups of exponential growth
% , 
% discrete subgroups of groups of Lie type~\cite{PS-proc}, 
and solvable groups of finite Pr\"ufer rank~\cite{PS-ems}. 
It also applies to various other specific groups for which 
mildly exponential bounds on $p(n)$ are known, 
such as the \emph{Baumslag--Solitar groups}
$\BS_q \ssu \GL(2,\qqq)$, $q\ge 2$, and the \emph{lamplighter groups}
$L_d = \zz_2 \wr \zz^d$, $d\ge 1$, see e.g.~\cite[$\S$15]{Woe}.
Let us emphasize that P-recursiveness fails for \emph{all}
symmetric generating sets in these cases.  In view of
Theorem~\ref{t:two-gen-sets}, the P-recursiveness fails
for \emph{some} generating sets of non-amenable groups containing
$F_2 \times F_2$.
This suggests that P-recursiveness of all generating sets is
a rigid property which holds for very few classes of group.
We conjecture that it holds for all  nilpotent groups.

% In fact, the underlying group matters very little in our
% argument.  By adding generators of $\SL(4,\zz)$ with even
% multiplicities one can make a Cayley graph with
% non-P-recursive probabilities of return.  This implies
% that P-recursive is not a group property, but depends also
% on a generating set.  It would be interesting to see such
% examples for symmetric generating sets as well.

\subsection{}\label{ss:fin-complexity}
\ts  Lemma~\ref{l:mod2} can be rephrased to say that the
\emph{subword complexity function} $c_{\bw}(n)  < 2^n$ for some~$n$
large enough (see e.g.~\cite{AS,BLRS}).  This is likely to be
far from optimal.  For example, for the Catalan numbers
$C_n = \frac1{n+1}\binom{2n}{n}$,
we have $\bw = 101000100000001\ldots$ In this case,
it is easy to see that the word complexity function $c_{\bw}(n) = \Theta(n)$,
cf.~\cite{DS}.   It would be interesting to find sharper upper
bounds on the maximal growth of $c_{\bw}(n)$, when $\bw$ is
the infinite parity word of a P-recursive sequence.  Note that
$c_{\bw}(n) = \Theta(n)$ for all automatic sequences~\cite[$\S$10.2]{AS},
and that the exponentially growing P-recursive
sequences modulo almost all primes are automatic provided
deep conjectures of Bombieri and Dwork, see~\cite{Chr}.

\subsection{}\label{ss:fin-asy}
\ts The integrality assumption in Theorem~\ref{t:asy}
cannot be removed as the following example shows.
Denote by $a_n$ the number of \emph{fragmented permutations},
defined as partitions of $\{1,\ldots,n\}$ into ordered lists of numbers
(see sequence {\tt A000262} in~\cite{OEIS}).  It is P-recursive since
$$
a_n \. = \. (2\ts n-1)\ts a_{n-1} \. - \. (n-1)(n-2)\ts a_{n-2} \quad \text{for all} \ \, \.n>2\ts.
$$
The asymptotics is given in~\cite[Prop.~VIII.4]{FS}:
$$
\frac{a_n}{n!} \, \sim \, \frac{1}{2\ts\sqrt{e\ts\pi}} \. e^{2\sqrt{n}} \. n^{-3/4}\ts.
$$
This implies that the theorem is false for the \emph{rational}, at most exponential
P-recursive sequence $\{a_n/n!\}$, since in this case we have mildly exponential terms.
To understand this, note that \ts $\sum_n \ts a_n \ts t^n/n!$ \ts is not a $G$-function
since the \ts $lcm$ \ts of denominators of $a_n/n!$ grow superexponentially.

\subsection{}\label{ss:fin-dfin}  \ts
Proving that a combinatorial sequence is not P-recursive is often difficult
even in the most classical cases.  We refer to~\cite{B+,BRS,BP,FGS,Kla,MR} for
various analytic arguments.  As far as we know, this is the first
proof by a computability argument.
% Curiously, we found further applications
% of Lemma~\ref{l:mod2} to disprove the Noonan--Zeilberger conjecture~\cite{GP}.

\vskip.5cm

{\small 
\noindent
\textbf{Acknowledgments:} \. We are grateful to Misha Ershov,
Martin Kassabov, Maxim Kontsevich, Andrew Marks, Marni Mishna,
Robin Pemantle, Bruno Salvy, Andy Soffer, Jed Yang and
Doron Zeilberger for interesting discussions.
Special thanks to Jean-Paul Allouche, Matthias Aschenbrenner,
Cyril Banderier, Alin Bostan, Mireille Bousquet-M\'{e}lou,
Martin Kassabov, Nick Katz, Christophe Pittet, 
Laurent Saloff-Coste and Richard Stanley, for many useful
remarks on the early draft of the paper, and help with the references.
The first author was partially supported by the University of California
Eugene V.~Cota-Robles Fellowship; the second author was partially
supported by the~NSF.

}

%\newpage

\end{document}